\documentclass[12pt]{article}

\usepackage{graphicx}
\usepackage{url}
\usepackage{amsmath}
\usepackage{amssymb}
\usepackage{microtype}
\usepackage{amsthm}
\usepackage[margin=1.5in]{geometry}
\usepackage[usenames]{color}

\newtheorem{thm}{Theorem}[section]

\newtheorem{cor}[thm]{Corollary}

\newtheorem{claim}[thm]{Claim}

\theoremstyle{definition}
\newtheorem{defn}[thm]{Definition}

\theoremstyle{definition}
\newtheorem{obsv}[thm]{Observation}

\theoremstyle{definition}
\newtheorem{question}[thm]{Question}

\newcommand{\rst}[1]{\ensuremath{{\mathbin\upharpoonright}%
\raise-.5ex\hbox{$#1$}}}

\begin{document}

\renewcommand{\thefootnote}{\fnsymbol{footnote}}

\title{The game of plates and olives}

\author{Teena Carroll\thanks{Mathematics Department, Emory \& Henry College, Emory VA; ccarroll@ehc.edu.}
\and
David Galvin\thanks{Department of Mathematics,
University of Notre Dame, Notre Dame IN; dgalvin1@nd.edu. Supported in part by National Security Agency grants H98230-10-1-0364 and H98230-13-1-0248 and by the Simons Foundation.}}

\date{\today}

\maketitle

\begin{abstract}
The game of plates and olives, introduced by Nicolaescu, begins with an empty table. At each step either an empty plate is put down, an olive is put down on a
plate, an olive is removed, an empty plate is removed, or the olives on two plates that both have olives on them are combined on one of the two plates, with the other plate
removed. Plates are indistinguishable from one another, as are olives, and there is an inexhaustible supply of each. 

The game derives from the consideration of Morse functions on the $2$-sphere. Specifically, the number of topological equivalence classes of excellent Morse functions on the $2$-sphere that have order $n$ (that is, that have $2n+2$ critical points) is the same as the number of ways of returning to an empty table for the first time after exactly $2n+2$ steps. We call this number $M_n$.

Nicolaescu gave the lower bound $M_n \geq (2n-1)!! = (2/e)^{n+o(n)}n^n$ and speculated that $\log M_n \sim n\log n$. In this note we confirm this speculation, showing that $M_n \leq (4/e)^{n+o(n)}n^n$.
\end{abstract}

\section{Introduction}

The basic aim of Morse theory is to gain knowledge of the topology of a manifold by studying smooth functions on it. For a smooth, compact, oriented manifold $X$ without boundary, an {\em excellent Morse function} is a smooth function $f:X\rightarrow {\mathbb R}$ whose critical points (points $x \in X$ where the differential of $X$ vanishes) are 
non-degenerate (the matrix of second partial derivatives is non-singular), and 
lie on distinct level sets.  

If $f$ has $m$ critical points $x_1, \ldots, x_m$ ordered such that $f(x_1) < \ldots < f(x_m)$, a {\em slicing} of $f$ is an increasing sequence $a_0, \ldots, a_m$ so that $a_0 < f(x_1) < a_1 < \ldots < a_{m-1} < f(x_m) < a_m$. 

\begin{defn} \label{defn-topoequiv}
Let $f$ and $g$ be excellent Morse functions on $X$ each with $m$ critical points, and let $a_0, \ldots, a_m$ and $b_0, \ldots, b_m$ be slicings of $f$ and $g$ respectively. 
Say that $f$ and $g$ are {\em topologically equivalent} if for each $i \in \{1,\ldots, m\}$ there is an orientation preserving diffeomorphism between the sublevel sets $\{x \in X:f(x) \leq a_i\}$ and $\{x \in X:g(x) \leq b_i\}$. 
\end{defn}

On the sphere $S^2$ an excellent Morse function has $2n+2$ critical points for some $n \geq 0$, of which exactly $n$ are saddle points (with the rest being either local minima or local maxima). Motivated by a question of Arnold \cite{Arnold}, in \cite{Nic06} Nicoleascu considered the question of the number of topological equivalence classes of excellent Morse functions on the $2$-sphere $S^2$ with $n$ saddle points. Let $T_n^2$ denote this value. He obtained the lower bound 
\begin{equation} \label{Mn-lb}
T_n^2 \geq (2n-1)!!=(2/e)^{n+o(n)} n^n
\end{equation}
so that $\liminf_{n\rightarrow \infty} \frac{\log T_n^2}{n\log n} \geq 1$.
He later speculated \cite{N-MO} that the lower bound is essentially correct, in that
\begin{equation} \label{eq-Nic-spec}
\log T_n^2 \sim n\log n
\end{equation}
as $n \rightarrow \infty$. The main purpose of this note is to validate this speculation in a strong way.
\begin{thm} \label{thm-Mn-ub}
We have $T_n^2 \leq (4/e)^{n+o(n)}n^n$
and so $\log T_n^2 \sim n\log n$ as $n \rightarrow \infty$.
\end{thm}
In the rest of the introduction we give some more background on the problem and describe Nicolaescu's game of plates and olives, which turns the problem into a purely combinatorial one. The proof of Theorem \ref{thm-Mn-ub} appears in Section \ref{sec-proof}, and we conclude with some questions in Section \ref{sec-questions}.

\subsection{Background on the geometrical equivalence problem}

Motivated by Hilbert's 16th problem calling for a study of the topology of real algebraic varieties, Arnold \cite{Arnold} raised the broad question of the possible structures of excellent Morse functions on various manifolds, and in particular on $S^n$, and the specific enumerative question of how the number of possible structures grows as a function of the number of critical points.

A notion of equivalence is needed to make precise sense of the enumerative question. Arnold used the following, which has subsequently obtained the name {\em geometrical equivalence}: two excellent Morse functions $f$ and $g$ on $X$ are {\em geometrically equivalent} if there are
orientation-preserving diffeomorphisms $r: X \rightarrow X$ and $\ell:
{\mathbb R} \rightarrow {\mathbb R}$ such that $g = \ell \circ f \circ r^{-1}$. 

If $f$ is an excellent Morse function on the circle $S^1$ then it must have a positive even number, say $2n+2$, of critical points. If these points are at $x_1, \ldots, x_{2n+2}$, read counter-clockwise around the circle starting from the global minimum $x_1$, then it must be 
that $x_i$ is a local minimum whenever $i$ is odd and a local maximum whenever $i$ is even. Associate with $f$ a permutation $\sigma_f$ of $\{1,\ldots,2n+2\}$ via the rule that when written in one-line notation the $i$th entry of $\sigma_f$ is the position of $f(x_i)$ when the $f(x_j)$'s are listed from smallest to largest. This permutation has a zig-zag property: entries in even positions are greater than their immediate neighbors, while entries in odd positions are less than their immediate neighbors. 

It is evident that two excellent Morse functions on $S^1$ are geometrically equivalent if they induce the same permutation, and that for every zig-zag permutation $\sigma$ of $\{1,\ldots,2n+2\}$ there is an excellent Morse function $f$ on $S^1$ with $\sigma_f=\sigma$. So the number $G_n^1$ of geometrical equivalence classes of excellent Morse functions on $S^1$ with $2n+2$ critical points is the same as the number of zig-zag permutations of $\{1,\ldots,2n+2\}$. This combinatorial sequence is well-known, see e.g. \cite[A000182]{Sloane}, and is closely related to the tangent function. Specifically, the Taylor series of $\tan x$ is $\sum_{n \geq 0} G_n^1 \frac{x^{2n+1}}{(2n+1)!}$. It follows that $\log G_n^1 \sim 2n \log n$ as $n \rightarrow \infty$. The sequence $(G_n^1)_{n \geq 0}$ begins $(1, 2, 16, 272, 7\!~936,\ldots)$.

On $S^2$, again an excellent Morse function has $2n+2$ critical points for some $n \geq 0$, of which exactly $n$ are saddle points. Evidently there is only one equivalence class with no saddle points (a representative example is the latitude function), and there are two classes with one saddle point (four critical points). A representative example of one of these is the height function of a landscape that has a global minimum elevation at the south pole ($D$), steadily rising elevation to around the north pole, and around the north pole has a two-peaked mountain (peaks $A$ and $B$) with a saddle point ($C$) between the peaks. See the picture on the left in Figure \ref{mountains-and-volcanos}. To get a representative example of the other class, replace the two-peaked mountain at the north pole with a volcano whose crater is bowl-shaped (lowest point $B$) and whose rim, viewed as a height function on the circle, has one local maximum ($A$) and one local minimum ($C$); the local maximum on the rim is the global maximum on $S^2$ and the local minimum on the rim is a saddle point on $S^2$. See the picture on the right in Figure \ref{mountains-and-volcanos}. Figure \ref{mountains-and-volcanos} is reproduced from \cite[Figures 1 and 2]{Arnold}, where these two landscapes are referred to as the Elbrus mountain and the Vesuvius volcano.    

\begin{figure}[ht!]
\begin{center}
\includegraphics[scale=.7]{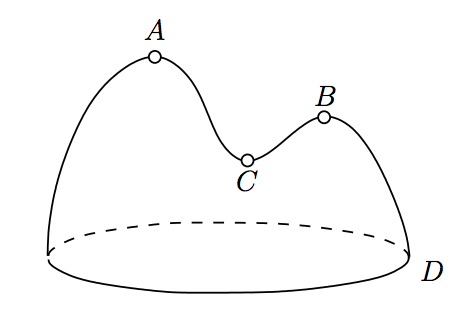}~~~\includegraphics[scale=.7]{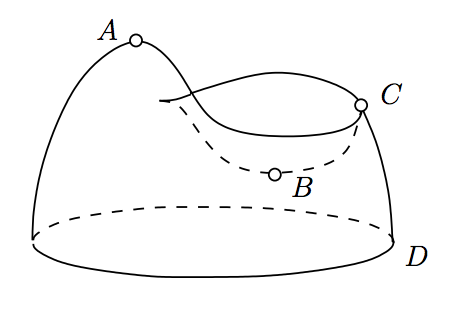}
\caption{The Elbrus mountain and the Vesuvius volcano.} \label{mountains-and-volcanos}
\end{center} 
\end{figure}

Arnold \cite{Arnold} calculated that the sequence $(G_n^2)_{n \geq 0}$ of geometrical equivalence classes of excellent Morse functions on $S^2$ with $n$ saddle points begins $(1, 2, 19, 428, 17\!~746, \ldots)$ \cite[A120420]{Sloane}, and speculated that the sequence grows like $n^{2n}$. In \cite{Nicolaescu, Nic06} Nicolaescu verified this speculation. He found recurrence relations for $G_n^2$, and also found that the generating function $\sum_{n \geq 0} G_n^2 \frac{x^{2n+1}}{(2n+1)!}$ can be expressed as the inverse of a function defined in terms of an (explicit) elliptic integral (note the similarity to the case of the circle, where we have $\arctan x = \int_0^x dt/(1+t^2))$. By analyzing this integral he concluded that $\log G_n^2 \sim 2n \log n$ as $n \rightarrow \infty$.

\subsection{Background on the topological equivalence problem}

The notion of topological equivalence (see Definition \ref{defn-topoequiv}) was introduced by Nicolaescu in \cite{Nicolaescu} (where he also introduced a notion of homological equivalence, that we do not consider here).
For a Morse function $f$ on $S^1$ with $2n+2$ critical points and a slicing $a_0, \ldots, a_{2n+2}$, the sublevel set $\{x\leq a_0\}$ is empty and $\{x \leq a_{2n+2}\}$ is all of $S^1$. For $i \in \{1,\ldots, 2n+1\}$ the set $\{x\leq a_i\}$ is a (non-empty) union of intervals, with $\{x\leq a_1\}$ and $\{x\leq a_{2n+1}\}$ both being a single interval, and for $i \in \{1,2n\}$ the number of intervals in $\{x\leq a_{i+1}\}$ is either one less than or one greater than the number in $\{x\leq a_i\}$. It follows that there is a one-to-one correspondence between equivalence classes and {\em proper Dyck paths} of semilength $n+1$ --- paths from $(0,0)$ to $(2n+2,0)$ taking steps $(1,1)$ and $(1,-1)$ that do not touch the $x$-axis except at $(0,0)$ and $(2n+2,0)$. The number of such paths is the $n$th Catalan number $C_n=\binom{2n}{n}/(n+1)$. The sequence of Catalan numbers begins $(1, 1, 2, 5, 14, \ldots)$ \cite[A000108]{Sloane} and its growth rate is given by $\log C_n \sim n\log 4$ as $n \rightarrow \infty$.         

On $S^2$, Nicolaescu obtained the lower bound (\ref{Mn-lb}) on the number $T_n^2$ of topological equivalence classes of excellent Morse functions with $n$ saddle points, and conjectured, via the asymptotic identity (\ref{eq-Nic-spec}), that there are many fewer topological than geometrical equivalence classes on $S^2$. Theorem \ref{thm-Mn-ub} validates this conjecture.

\subsection{The game of plates and olives}

Our proof of Theorem \ref{thm-Mn-ub} passes through the game of plates and olives, a purely combinatorial approach to the question of topological equivalence of Morse functions on $S^2$ described in \cite{Nic06}.

If $a \in {\mathbb R}$ is such that there is no critical point $x \in X$ with $f(x)=a$, then in \cite{Nicolaescu} it is shown that the sublevel set $\{f\leq a\}$ is either the whole of $S^2$ or is topologically equivalent to a finite set of disjoint disks (for convenience we refer to these as {\em outer} disks), with some of the outer disks perhaps being modified by the removal of a finite set of disjoint disks from the interior ({\em inner} disks). As the value $a$ crosses the critical point $x_1$, the sublevel set changes from being empty to being a single disk. As it crosses the critical point $x_{2n+2}$ (the final critical point) it changes from being a disk to being the whole of $S^2$. Crossing each other critical point, the topology of the sublevel set changes in one of the following ways. Either
\begin{itemize}
\item a new outer disk appears, or
\item an new inner disk appears in an outer disk, or
\item an inner disk disappears from an outer disk, or
\item an outer disk with no inner disks disappears, or
\item two outer disks, each with some inner disks, merge (and the number of inner disks in the resulting outer disk is the sum of the numbers of inner disks in the two outer disks that merged).
\end{itemize}
(See \cite[Section 3]{Nicolaescu} for a detailed discussion).
The evolution of the topology of the sublevel sets, which fully determines the topological equivalence class of $f$, can thus be encoded by the game of {\em plates and olives}, which proceeds as follows. There are inexhaustible supplies of indistinguishable olives and indistinguishable plates. The game begins with an empty table (corresponding to the empty sublevel set, before the first critical value is reached). At each step, one of the following happens. Either 
\begin{itemize}
\item an empty plate is put down (a {\em plate-add} or $P^+$ move), or 
\item an olive is put down on an existing plate (an {\em olive-add} or $O^+$ move), or
\item an olive is removed, necessarily from an existing plate (an {\em olive-remove} or $O^-$ move), or
\item an empty plate is removed (a {\em simple plate-remove} or $P^-_s$ move), or 
\item the olives on two plates which both have olives on them are combined on one of the two plates (it doesn't matter which since plates are indistinguishable) with the other (now empty) plate removed (a {\em complex plate-remove} or $P^-_c$ move).
\end{itemize}
These steps correspond directly to the changes in the topology of the sublevel sets of $f$ as the critical values are crossed.
The game ends the first time a step results in a return to an empty table (this corresponds to the final transition of the sublevel sets, from disk to $S^2$). Necessarily a game begins with a $P^+$ move and ends with a $P^-_s$ move. If there are a total of $n$ add moves, not counting the initial $P^+$ move,  then we refer to $n$ as the {\em length} of the game. Note that this necessitates that there are also exactly $n$ remove moves (not counting the final $P^-_s$ move). We denote by $M_n$ the number of games of length $n$. 

As discussed earlier, Nicolaescu \cite[Section 7]{Nicolaescu} exhibits a bijective correspondence between topological equivalence classes of excellent Morse function on the $2$-sphere $S^2$ with $n$ saddle points, and games of plates and olives of length $n$, so that $T_n^2=M_n$ and we can prove Theorem \ref{thm-Mn-ub} by showing $M_n \leq (4/e)^{n+o(n)}$.

Evidently $M_0 = 1$ and $M_1=2$. One of the two games of length $1$ --- the one corresponding to the Vesuvius volcano --- consists of two $P^+$ moves followed by two $P^-_s$ moves, and the other --- corresponding to the Elbrus mountain --- consists of a $P^+$ move followed by an $O^+$ followed by an $O^-$ followed by a $P^-_s$. Weingartner \cite{Weingartner} has calculated $M_n$ for $n \leq 18$; the sequence $(M_n)_{n \geq 0}$ begins $(1, 2, 10, 76, 772, \ldots)$ \cite[A295929]{Sloane}.

%We have also calculated that $M_2=10$, $M_3=76$, $M_4=772$ and $M_5=9856$. This calculation was done by creating a directed graph whose vertices are all possible non-empty configurations of plates and olives with no more than six plates and olives in total (there are 29 such) and with an edge from $a$ to $b$ if configuration $b$ can be reached from configuration $a$ via a single move. From this it is easy to read off the number of walks of any given length (up to 10) from the single-plate state to itself, by looking the appropriate entry of an appropriate power of the adjacency matrix of the graph. 

The game of plates and olives can also be viewed as a walk on the set of integer partitions. A configuration in the game can be encoded as a vector $\langle a_1, \ldots, a_k\rangle$ with $a_1 \geq a_2 \ldots \geq a_k > 0$, where an entry $a_j$ in the vector corresponds to a plate with $a_j-1$ olives. Such a vector is referred to as a {\em partition} of the integer $a_1+\ldots+a_k$. Legitimate moves in the game can easily be encoded as directed edges in a graph whose vertex set consists of all partitions of all positive integers, together with the empty partition $\emptyset$ (the unique partition of $0$).
Evidently $M_n$ equals the number of directed walks in this graph from $\emptyset$ to itself of length $2n+2$ that do not involve $\emptyset$ except at the beginning and end. (The values of $M_n$ for $n \leq 18$ have been calculated by considering powers of the adjacency matrix of relevant subgraphs of this graph.)   

We take the opportunity here to correct a slight error in \cite{Nicolaescu}, where the lower bound (\ref{Mn-lb}) is presented as $M_n \geq (2n+1)!!$. If we ignore edges in the directed partition graph that correspond to $P^-_c$ moves, then the resulting graph is un-directed (in the sense that whenever there is an edge from $a$ to $b$ there is a corresponding edge from $b$ to $a$), and is in fact the Hasse diagram of the very well studied Young's lattice. It is known (see for example \cite{Stanley}) that there are exactly $(2n+1)!!$ walks of length $2n +2$ in Young's lattice that start and end at $\emptyset$, and it was this observation that led to the lower bound $M_n \geq (2n+1)!!$ in \cite{Nicolaescu}. This count includes walks that make interim returns to $\emptyset$, however, and these do not correspond to legitimate games of plates and olives (or Morse functions). The bound is easily salvaged, though: there are $(2n-1)!!$ walks of length $2n$ in Young's lattice that start and end at $\emptyset$, and we can correspond to each one a unique walk of length $2n+2$ that starts and ends at $\emptyset$ and that does not otherwise involve $\emptyset$. We do this by adding the step from $\emptyset$ to $\langle 1 \rangle$ at the beginning of the walk, adding the step from $\langle 1 \rangle$ to $\emptyset$ at the end, and extending the partition at all remaining steps by including one new $1$. This yields the bound (\ref{Mn-lb}). This issue of interim returns to $\emptyset$ also accounts for the discrepancy between the values we quote for $M_2$, $M_3$ and $M_4$ ($10$, $76$ and $772$) and those quoted in \cite{Nicolaescu} ($15$, $107$ and $981$).   

\section{Proof of Theorem \ref{thm-Mn-ub}} \label{sec-proof}

It will be convenient to refine the notion of an olive-add move:
\begin{itemize}
\item when an olive is put down on an empty plate we refer to this as a {\em first olive-add} or $O^+_f$ move, and 
\item when an olive is put down on a plate that already has olives on it we refer to this as a {\em later olive-add} or $O^+_l$ move.
\end{itemize}
Let ${\mathcal M}_n$ be the set of all games of length $n$. For $g \in {\mathcal M}_n$ the {\em skeleton} $S(g)$ of the game is the sequence of length $2n+2$ whose $i$th term is $P^+$, $O^+_f$, $O^+_l$, $P^-_s$, $P^-_c$, $O^-$ according to whether at the $i$th step the move made is plate-add, first olive-add, later olive-add, simple plate-remove, complex plate-remove, or olive-remove. Note that the skeleton begins with $P^+$ and ends with $P^-_s$. 

From $S(g)$ we can read off a number of other parameters that it will be useful to associate with $g$: 
\begin{itemize}
\item $v^+_f(g)$, the number of $O^+_f$ moves in the game $g$, (we use $v$ for oli{\em v}e to avoid confusion with Landau's $o$-notation)
\item $v^+_l(g)$, the number of $O^+_l$ moves in $g$,
\item $v(g):=v^+_f(g)+v^+_l(g)$, which equals both the number of olive-add (first or later) and the number of olive-remove moves,
\item $p^-_s(g)$, the number of $P^-_s$ moves, not counting the $(2n+2)$nd (the final) move,
\item $p^-_c(g)$, the number of $P^-_c$ moves, and
\item $p(g):=p^-_s(g)+p^-_c(g)$, which equals both the number of plate-remove (simple or complex) and the number of plate-add moves, not counting the first and the $(2n+2)$nd.
\end{itemize}
We clearly have the relation $v(g)+p(g)=n$. We also have the following relation, which will be crucial later:
\begin{obsv} \label{obsv-crucial}
$p^-_c(g) \leq v^+_f(g)$. 
\end{obsv} 
Indeed, each plate removed in a complex plate-remove move must at least once have had an olive placed on it while it was empty.

Our intuition is that there should not be too many games that involve many $P^-_s$ moves --- once we have decided to make a $P^-_s$ move, there is a unique choice for the specific move, so $P^-_s$ moves are not ``costly''. Similarly there should not be too many games that involve many $O^+_f$ moves, and hence (by Observation \ref{obsv-crucial}) not too many that involve many $P^-_c$ moves. After verifying this intuition we are able to focus attention on games in which the number of plates involved is relatively small; we will put an upper bound on the number of such games by translating the problem to a classical one of enumerating weighted Dyck paths.

We will initially put an upper bound on $M_n$ by bounding, for each possible skeleton $S$, the number of games that can have that skeleton. The key observation is the following. Suppose that at some moment there are exactly $t$ olives on the table. Encode the state of the game by a vector $(a_0, a_1, a_2, \ldots)$ where $a_i$ is the number of plates that have exactly $i$ olives on them. Let $w(t)$ be the largest integer $t'$ such that $t'(t'-1)/2 \leq t$. Then we claim that
\begin{equation} \label{eq-w}
|\{i: a_i \neq 0\}| \leq w(t).
\end{equation}
Indeed, if $|\{i: a_i \neq 0\}| = \ell$ then the number of olives on the table has to be at least
$$
0+1+\ldots + (\ell-1) = \frac{\ell(\ell-1)}{2}
$$ 
(to obtain $|\{i: a_i \neq 0\}| = \ell$ with as few olives as possible one needs one plate with $i$ olives, for each $i \in \{0, \ldots, \ell-1\}$). If $\ell > w(t)$ then by the definition of $w(t)$ this requires more than $t$ olives.

Using the indistinguishability of plates and of olives, (\ref{eq-w}) allows us to put an upper bound on the number of possible moves of various types that can happen at any particular moment, as a function of $t$ and $w(t)$. 
\begin{claim} \label{claim-precise}
If at some moment there are exactly $t$ olives on the table, then there are
\begin{itemize}
\item at most $1$ $O^+_f$, $P^+$ and $P^-_s$ moves that can be made,
\item at most $w(t)$ $O^+_l$ moves,
\item at most $w(t)-1$ $O^-$ moves, and
\item at most $w^2(t)$ $P^-_c$ moves.
\end{itemize}    
\end{claim}

\begin{proof}
The first two points are straightforward. For the third point above, a bound of $w(t)$ is clear, but we can refine the argument slightly and replace $w(t)$ with $w(t)-1$. Indeed, if there are exactly $t$ olives on the table and at least one empty plate, then there are at most $w(t)-1$ $O^-$ moves that can be made (since we cannot remove an olive from an empty plate). On the other hand if there are no empty plates, 
then again there can be at most $w(t)-1$ possible $O^-$ moves (if we could make $w(t)$ $O^-$ moves, then with the addition of an empty plate we would have a configuration with $t$ olives and  $w(t)+1$ possible $O^-$ moves, a contradiction).

For the fourth point note that in this case we have to choose two plates with some olives on them.  
\end{proof}
The difference between $w(t)$ and $w(t)-1$ in the third point is asymptotically inconsequential but will be computationally helpful later.  In the fourth point, since the order of the plates does not matter we could replace $w^2(t)$ with $w^2(t)/2$, but this gains us nothing substantial.

We have $w(t) \sim \sqrt{2t}$ as $t \rightarrow \infty$ and so $w(t) \leq \sqrt{3t}$ for all sufficiently large $t$. Since for any game in ${\mathcal M}_n$ there can be at most $n$ olives on the table at any given moment, we have the following crude estimates as a corollary of Claim \ref{claim-precise}:
\begin{cor} \label{cor-crude}
For all sufficiently large $n$, at any moment
\begin{itemize}
\item at most $\sqrt{3n}$  $O^+_l$ or $O^-$ moves can be made,
\item at most $1$ $P^+$, $P^-_s$ or $O^+_f$ moves, and
\item at most $3n$ $P^-_c$ moves.
\end{itemize}    
\end{cor}

Using Corollary \ref{cor-crude} we can quickly establish an upper bound of the form $M_n \leq C^n n^n$ (so that indeed $\log M_n \sim n\log n$, as speculated by Nicolaescu \cite{N-MO}).

\begin{claim} \label{claim-crude}
$M_n \leq 108^n n^n$.
\end{claim}

\begin{proof} 
The skeleton of a game always starts with a $P^+$ and ends with a $P^-_s$. Each of the remaining $2n$ steps can be any one of $P^+$, $O^+_f$, $O^+_l$ $P^-_s$, $P^-_c$ or $O^-$, so there are at most $36^n$ skeletons. 
If a skeleton has $v$ steps that are $O^+$ (either $O^+_f$ or $O^+_l$) then it also has $v$ steps that are $O^-$, and there at most $\big(\sqrt{3n}\big)^{2v}=3^v n^v$ choices for the particular moves that are made at those steps. Of remaining $2n-2v+2$ steps, the first and last are forced, as are the $n-v$ $P^+$ steps, while on the $n-v$ remaining $P^-_s$ or $P^-_c$ steps there are at most $(3n)^{n-v}=3^{n-v}n^{n-v}$ choices for the particular moves that are made. This means that each possible skeleton can be the skeleton of at most $3^n n^n$ games, and so $M_n \leq 108^n n^n$.
\end{proof}

To reduce the constant $C$ from  $108$ to $4/e$, we have to work harder. We begin by counting the games in which $p^-_s+p^-_c \geq 16n/\log n$. The goal is to show that there are relatively few games with this property, by obtaining an upper bound of the form $C^n n^n$ for the number of such games, with $C < 2/e$ (so that this upper bound is smaller than the lower bound on $M_n$ from (\ref{Mn-lb})). We will argue that we can reduce the bound of Claim \ref{claim-crude} by $(3n)^{p_s^-}$ by using that there is only one option for a $P_s^-$ move, rather than at most $3n$, and we can reduce the bound also by $(\sqrt{3n})^{v_f^+} \geq (\sqrt{3n})^{p_c^-}$ by using that there is only one option for an $O_f^+$ move, rather than at most $\sqrt{3n}$. If $p^-_s+p^-_c$ is suitably large, this reduces the count to below $(2/e)^n n^n$.

There are at most $(n+1)^2$ choices for the pair $(p^-_s,p^-_c)$ in the regime $p^-_s+p^-_c \geq 16n/\log n$, and for each such pair there are again at most $36^n$ skeletons. For a given skeleton
\begin{itemize}
\item there is $1$ way to specify the move made at each of the $p^-_s+p^-_c+1$ steps labelled $P^+$ (the $+1$ here is included for the first step of the game, a $P^+$ move),
\item among the $v:=n-(p^-_s+p^-_c)$ steps labelled $O^+$ (either $O^+_f$ or $O^+_l$), at least $p^-_c$ of them must be $O^+_f$ (by Observation \ref{obsv-crucial}). There is $1$ way to specify the move made at each of these at least $p^-_c$ steps, and at most  $\sqrt{3n}$ ways to specify the move made at each of the remaining at most $v-p^-_c$ steps,
\item there are at most  $\sqrt{3n}$ ways to specify the move made at each of the $v$ steps labelled $O^-$,
\item there are at most  $3n$ ways to specify the move made at each of the $p^-_c$ steps labelled $P^-_c$, and 
\item there is $1$ way to specify the move made at each of the $p^-_s+1$ steps labelled $P^-_s$ (the $+1$ here is included for the last step of the game, a $P^-_s$ move).
\end{itemize} 
Putting this together we find that each possible skeleton can be the skeleton of at most
\begin{eqnarray*}
(3n)^{\frac{v-p^-_c}{2} + \frac{v}{2} + p^-_c} & = & (3n)^{n-p^-_s - \frac{p^-_c}{2}}  \\
& \leq & 3^n n^{n-\left(\frac{p^-_s+p^-_c}{2}\right)} \\
& \leq & \frac{3^n n^n}{n^\frac{8n}{\log n}} \\
& = & \left(\frac{3}{2^8}\right)^n n^n
\end{eqnarray*}
games, and so the number of games in this regime ($p^-_s+p^-_c \geq 16n/\log n$) is at most
\begin{equation} \label{inq-many_plates}
(n+1)^2\left(\frac{108}{2^8}\right)^n n^n \leq \left(\frac{2}{e}-c\right)^{n+o(n)} n^n 
\end{equation}
where $c>0$ is a constant.

We now turn to the regime $p^-_s+p^-_c \leq 16n/\log n$. Here we will not distinguish between $P^-_s$ and $P^-_c$ moves, so we refer to both as $P^-$ moves; nor will distinguish between $O^+_f$ and $O^+_l$ moves, so we refer to both simply as $O^+$ moves. This convention carries through to the notion of the skeleton: we will now only use symbols $O^+$, $O^-$, $P^+$ and $P^-$ in the skeleton, and we will work with the parameters $p(g)$ and $v(g)$ rather than the more refined parameters $p^-_s$, $p^-_c$, $v^+_f$ and $v^+_l$.

We count games slightly differently in this regime. To each game $g$ we associate a Dyck path of semilength $v(g)$ (a path from $(0,0)$ to $(2v(g),0)$ taking steps $(1,1)$ and $(1,-1)$ that stays on or above the $x$-axis). The association is made by taking a $(1,1)$ step each time an $O^+$ move is made and a $(1,-1)$ step each time an $O^-$ move is made. We refer to this as the {\em underlying olive Dyck path} ${\rm OD}(g)$ of the game.

We begin by fixing a $p \leq 16n/\log n$ (there are at most $16n/\log n + 1=2^{o(n)}$ options for $p$). Next we select a Dyck path ${\rm OD}$ of semilength $v:=n-p$ to be the underlying olive Dyck path of the game. There are
$$
\binom{2n}{2v}2^{2p} = \binom{2n}{2p} 4^p = 2^{o(n)}  
$$  
ways to extend this underlying olive Dyck path to a skeleton. The $\binom{2n}{2v}$ accounts for locating where in the $2n$ internal steps of the skeleton (those that are not the first or the last) the $2v$ steps labelled $O^+$ and $O^-$ are located, and the $2^{2p}$ accounts for deciding which among the remaining $2p$ internal steps are labelled $P^+$ and which $P^-$. In the second equality, we use the symmetry of binomial coefficients. In the asymptotic estimate we use the basic bound $\binom{n}{k}\leq (n/k)^k$ along with $p\leq 16n/\log n$.

To count at most how many games can have a particular skeleton we now use the more precise Claim \ref{claim-precise} rather than the cruder Corollary \ref{cor-crude}. There is $1$ way to specify which $P^+$ move happens at each step labelled $P^+$ in the skeleton, and at most $3n$ ways to specify which $P^-$ move happens at each step labelled $P^-$, so there are a total of at most $(3n)^p = 2^{o(n)} n^p$ ways to specify which $P^+$, $P^-$ moves happen in the game.    

We now turn to specifying the $O^+$ and $O^-$ moves. At a step labelled $O^+$ that is at height $t$ in ${\rm OD}$ (height measured by the $y$ co-ordinate of the lower end of the step in the path; so the height is equal to the number of olives that are on the table at that moment), there are at most $w(t)$ ways to specify which $O^+$ move happens. At a step labelled $O^-$ that is at height $t$ in ${\rm OD}$, there are  $t+1$ olives on the table, and so there are at most $w(t+1)-1$ ways to specify which $O^-$ move happens. 

It is straightforward to pair up steps labelled $O^+$ and $O^-$ in ${\rm OD}$, with paired steps having the same height. On such a pair at height $t$, the number of ways of specifying which moves happen is at most 
$$
w(t)(w(t+1)-1) \leq w(t+1)(w(t+1)-1) \leq 2(t+1)
$$ 
(the first inequality is by monotonicity of $w$ and the second follows from the definition of $w$). Denoting the height of a step $s$ by $h(s)$ it follows that there are at most
$$
\prod_{\mbox{$s$ an $O^+$ step of ${\rm OD}$}} 2(h(s)+1) \leq 2^n  \prod_{\mbox{$s$ an $O^+$ step of ${\rm OD}$}} (h(s)+1)
$$   
ways to specify which $O^+$, $O^-$ moves happen in the game, and so the number of games (for each $p$) is at most
\begin{equation} \label{eq-Dyck-sum}
2^{n+o(n)} n^p \sum \left\{ \prod_{\mbox{$s$ an $O^+$ step of ${\rm OD}$}} (h(s)+1):\begin{array}{c}
\mbox{${\rm OD}$ a Dyck path}\\
\mbox{of semilength $n-p$}
\end{array}
\right\}.
\end{equation}
It is well-known (see for example \cite[Section 2.6]{Callan}) that the sum in (\ref{eq-Dyck-sum}) evaluates exactly to $(2(n-p)-1)!!$, so that the number of games $g$ with $p(g)=p \leq 16n/\log n$ is at most
$$
2^{n+o(n)} n^p  \left(\frac{2}{e}\right)^{n-p+o(n-p)} (n-p)^{n-p} = \left(\frac{4}{e}\right)^{n+o(n)} n^n.
$$
Given that there at most $2^{o(n)}$ choices for $p$ and combining with (\ref{inq-many_plates}) this completes the proof of Theorem \ref{thm-Mn-ub}.

\section{Questions} \label{sec-questions}

Together (\ref{Mn-lb}) and Theorem \ref{thm-Mn-ub} show that $(1/n)\sqrt[n]{M_n}=\Theta(1)$.
\begin{question} \label{quest-lim}
Does $\lim_{n \rightarrow \infty} (1/n)\sqrt[n]{M_n}$ exists, and if so where in the interval $[2/e,4/e]$ does it lie?
\end{question}
Both the upper and lower bounds seem to leave room for improvement --- the lower bound because it ignores $P^-_c$ moves and the upper bound because it uses a worst-case bound $\sqrt{2m}$ on the number of distinct parts in a partition of the integer $m$, whereas (as shown by Wilf \cite{Wilf}) the typical partition has only $(1+o(1))\sqrt{6m}/\pi$ distinct parts. A calculation of $M_n$ for $n \leq 18$ \cite{Weingartner} shows that $(1/n)\sqrt[n]{M_n}$ is decreasing in this range with $(1/18)\sqrt[18]{M_{18}} \approx 1.09206 < 4/e$, providing further evidence that the upper bound is not tight, as well as providing evidence that the limit in Question \ref{quest-lim} does exist. 

The study of geometrical and topological equivalence classes of Morse functions on $S^1$ and $S^2$ leads naturally to three well-studied combinatorial structures --- zig-zag permutations, Dyck paths and Young's lattice --- and it would be of interest to see what structures emerge when the same problems are tackled on $S^n$ for $n \geq 3$, or on other manifolds, such as the torus.

\end{document}